\documentclass{article}
\usepackage{amssymb,amsmath}
\usepackage[curve]{xypic}
\usepackage[amsmath,hyperref,thmmarks]{ntheorem}

\title{Noncommutativity as a colimit}
\author{Benno van den Berg\thanks{Technische Universit{\"a}t
    Darmstadt, Fachbereich Mathematik.}
    ~and Chris Heunen\thanks{Oxford University Computing
    Laboratory, supported by the Netherlands Organisation for
    Scientific Research (NWO). Part of this work was performed while
    the author visited the IQI at Caltech.}}

\newcommand{\after}{\circ}
\newcommand{\cat}[1]{\ensuremath{\mathbf{#1}}}
\newcommand{\Cat}[1]{\ensuremath{\mathbf{#1}}}
\newcommand{\id}[1][]{\ensuremath{\mathrm{id}_{#1}}}
\newcommand{\op}{\ensuremath{^{\mathrm{op}}}}
\newcommand{\cotuple}[2]{\ensuremath{[ #1\,,\,#2 ]}}

\newcommand{\tensor}{\ensuremath{\otimes}}
\newcommand{\field}[1]{\ensuremath{\mathbb{#1}}}
\newcommand{\ie}{\textit{i.e.}\ }

\newcommand{\cC}{\ensuremath{\mathcal{C}}}

\newcommand{\generated}[2]{\ensuremath{#1 \langle #2 \rangle}}
\newcommand{\Generated}[2]{\ensuremath{#1 \langle\hspace*{-2pt}\langle #2 \rangle\hspace*{-2pt}\rangle}}
\newcommand{\eqcomment}[1]{\ensuremath{\tag*{\mbox{\scriptsize{#1}}}}}
\newcommand{\commeas}{\ensuremath{\odot}}
\newcommand{\nul}{\ensuremath{\textbf{0}}}
\newcommand{\een}{\ensuremath{\textbf{1}}}
\newcommand{\Proj}{\ensuremath{\mathrm{Proj}}}
\newcommand{\colim}{\ensuremath{\mathop{\mathrm{colim}}}}
\newcommand{\norm}[1]{\left|\!\left| #1 \right|\!\right|}
\newcommand{\RP}{\ensuremath{\text{RP}}}

\newdir{ >}{{}*!/-5pt/@{>}}

\theoremnumbering{arabic}
\theoremstyle{plain}
\theorembodyfont{\itshape}
\theoremheaderfont{\normalfont\bfseries}
\theoremseparator{}
\newtheorem{theorem}{Theorem}
\newtheorem{lemma}[theorem]{Lemma}
\newtheorem{proposition}[theorem]{Proposition}
\newtheorem{corollary}[theorem]{Corollary}
\theorembodyfont{\normalfont}
\newtheorem{definition}[theorem]{Definition}

\newtheorem{remark}[theorem]{Remark}
\theoremstyle{nonumberplain}
\theoremheaderfont{\scshape}
\newtheorem{proof}{Proof}
\qedsymbol{\ensuremath{\Box}}

\newif\ifbijnotfirst
\global\def\bijprev{}
\newbox\widerone
\newbox\widertwo
\newdimen\widerdimen
\newenvironment{bijectivecorrespondence}
  {\global\bijnotfirstfalse%
   \begin{array}{cl}}
  {\end{array}}
\newcommand{\makerule}[2]{%
  \setbox\widerone\hbox{\ensuremath{#1}}%
  \setbox\widertwo\hbox{\ensuremath{#2}}%
  \ifdim\wd\widerone>\wd\widertwo
    \widerdimen=\wd\widerone
  \else
    \widerdimen=\wd\widertwo
  \fi%
  \vbox{%
    \kern1pt%
    \hrule width \widerdimen height \arrayrulewidth depth 0pt%
    \kern1pt%
    \hrule width \widerdimen height \arrayrulewidth depth 0pt%
    \kern1pt%
  }}
\newcommand{\correspondence}[2][]{%
  \crcr%
  \ifbijnotfirst%
    \omit\hfil\makerule{#2}{\bijprev}\hfil\cr%
  \else
  \fi%
  \global\bijnotfirsttrue%
  \global\def\bijprev{#2}%
  \ensuremath{#2} & \mbox{#1}}

\begin{document}
\maketitle

\begin{abstract}
\noindent
  We give substance to the motto ``every partial algebra is the colimit of
  its total subalgebras'' by proving it for partial Boolean algebras (including
  orthomodular lattices), the new notion of partial C*-algebras
  (including noncommutative C*-algebras), and variations such as
  partial complete Boolean algebras and partial AW*-algebras. Both
  pairs of results are related by taking projections. As corollaries we
  find extensions of Stone duality and Gelfand duality.
  Finally, we investigate the extent to which the Bohrification
  construction~\cite{heunenlandsmanspitters:bohrification}, that works
  on partial C*-algebras, is functorial.
\end{abstract}

\section{Introduction}
\label{sec:intro}

This paper is intended as a contribution to the Bohrification
programme~\cite{heunenlandsmanspitters:bohrification}, which
tries to give a mathematically precise expression to Bohr's doctrine
of classical concepts, saying that a quantum mechanical system is
to be understood through its classical fragments. On Bohr's view, as
understood within this programme, quantum mechanical systems do not
allow a `global' interpretation as a classical system, but they
do so `locally'. The programme essentially makes two claims about
these classical `snapshots'. Firstly, it is only through these
snapshots that the behaviour of a system and physical reality can be
understood. Secondly, these snapshots contain all the information
about the system that is physically relevant. 

The main result of~\cite{heunenlandsmanspitters:bohrification}
(inspired by earlier work of Butterfield, Isham and
D{\"o}ring~\cite{doeringisham:thing}) is that the collection of these
classical snapshots can be seen as forming a single classical system in a
suitable topos -- its so-called Bohrification; we will briefly recall the
details of this construction in Section~\ref{sec:bohrification}. The
implication of this result is that a quantum mechanical 
system can be seen as a classical one, if one agrees that nothing
physically relevant is lost by considering classical snapshots and if
one is willing to change the logic from a classical into an
intuitionistic one. 
The question is left open as to how strong these premises
are: how much, if anything, of the information about the
original quantum mechanical system is lost by conceding in this way to
consider it as a classical system? This article investigates how much
information about a quantum mechanical system can be reconstructed
from its Bohrification by means of colimits. The fact that quantum
mechanical systems can be modelled as noncommutative algebras, and
classical systems as commutative ones, explains the title
``noncommutativity as a colimit''.

Our conceptual contributions to the Bohrification programme are
twofold. First, we contend that the programme is most naturally seen
as concerned with \emph{partial algebras}, \ie sets on which algebraic
structure is only partially defined. Such partial algebras are equipped
with a binary (`commeasurability') relation, which holds between two
elements whenever they can be elements of a single
classical snapshot. In this setting, commeasurability is often also
called commutativity, since in typical examples commeasurability means
commutativity in a totally defined algebraic structure. We will show
that a meaningful theory of such partial algebras can be
developed. In particular, and this is our second contribution, we will
indicate why Bohrification makes sense for partial algebras 
and use this to investigate the functorial aspects of this
construction. 
 
The idea to consider partial algebras is not new. In fact, in their
classic paper~\cite{kochenspecker:hiddenvariables}, Kochen and Specker
use the language of partial algebras to state their famous result,
saying that many algebras occurring in quantum mechanics cannot be
embedded (as partial algebras) into commutative ones (see also
\cite{redei:quantumlogic}). Their interpretation of this fact as
excluding a hidden variable interpretation of quantum mechanics has
remained somewhat controversial. In the Bohrification programme, this
result is taken just as a mathematical confirmation of the view that
quantum mechanical systems do not allow for global interpretations as
classical systems.  

The result by Kochen and Specker, together with the fact that
Bohrification works for partial algebras, indicates that the
Bohrification programme is most naturally developed in the context of
partial algebras. This led us to develop a new notion of `partial
C*-algebra', to the study of which most of this paper is devoted. Our
main technical result is that such a partial 
C*-algebra is the colimit of its total (commutative) subalgebras,
which explains the relation between a partial algebra and its
classical snapshots in categorical terms. % (In this sense, the motto
% ``noncommutativity as a colimit'' should really be
% ``incommeasurability as a colimit''. However, the latter does not convey
% the context of the Bohrification programme well.)

In more detail, the contents of this paper are as follows. First we
consider Kochen and Specker's notion of a partial Boolean algebra, as
a kind of toy example, and prove our main result for them in 
Section~\ref{sec:bool}, as well as for the variation of partial
complete Boolean algebras. This also makes precise the widespread intuition
that an orthomodular lattice is an amalgamation of its Boolean blocks.
In this simple setting the theorem already has interesting corollaries, as
it allows us to derive an adjunction extending Stone duality in
Section~\ref{sec:stone}. But our main interest lies with partial
C*-algebras, which Section~\ref{sec:cstar} studies; we also prove the
main result for variations, such as AW*-algebras and Rickart
C*-algebras in that section. This enables an adjunction
extending Gelfand duality and puts the Kochen-Specker theorem in
another light in Section~\ref{sec:gelfand}. The two parallel settings
of Boolean algebras and C*-algebras are related by taking projections,
as Section~\ref{sec:projections} discusses. Finally,
Section~\ref{sec:bohrification} shows that the Bohrification
construction works for partial C*-algebras as well: we investigate its
functorial properties, and conclude that its essence is in fact (a
two-dimensional version of) the colimit theorem.\footnote{Both the
  categories of Boolean algebras and of commutative C*-algebras are
  algebraic, \ie monadic over the category of
  sets~\cite{johnstone:stonespaces}. We expect that the definitions
  and results of the present article can be extended to a more general
  theory of partial algebra, but refrain from doing so because the two
  categories mentioned are our main motivation.}

\section{Partial Boolean algebras}
\label{sec:bool}

We start with recalling the definition of a partial Boolean algebra, as
introduced by Kochen and Specker~\cite{kochenspecker:hiddenvariables}.

\begin{definition}
  A \emph{partial Boolean algebra} consists of a set $A$ with
  \begin{itemize}
    \item a reflexive and symmetric binary
      (\emph{commeasurability}) relation $\commeas \subseteq A \times A$;
    \item elements $0,1 \in A$;
    \item a (total) unary operation $\lnot \colon A \to A$;
    \item (partial) binary operations $\land, \lor \colon \commeas \to A$;
  \end{itemize}
  such that every set $S \subseteq A$ of pairwise commeasurable elements is contained in a set $T \subseteq A$, whose elements are also pairwise commeasurable, and on which the above operations determine a Boolean algebra structure.\footnote{Note that this means that $T$ must contain 0 and 1 and has to be closed under $\lnot, \land$ and $\lor$.\label{footnote:subalgebrastructural}}

  A morphism of partial Boolean algebras is a function that
  preserves commeasurability and all the algebraic structure, whenever
  defined. We write $\Cat{PBoolean}$ for the resulting category.
\end{definition}

Clearly, a partial Boolean algebra whose commeasurability relation is total is nothing but a Boolean algebra. 
%Because of the correspondence
%between Boolean algebras and Boolean rings, commeasurability is also
%called commutativity by many authors in this context.
For another example, if we declare two elements $a,b$ of an
orthomodular lattice to be commeasurable when $a = (a \land b) \lor
(a \land b^\perp)$, as is standard, any orthomodular lattice is
seen to be a partial Boolean algebra.
% Recall that an orthocomplemented lattice $X$ is called
% \emph{orthomodular} when $x \leq y$ implies $x \land (x^\perp \lor y)
% = y$.
In this case the above observation about totality becomes a known fact:
an orthomodular lattice is a Boolean algebra if and only if any pair of elements is commeasurable~\cite{kalmbach:orthomodularlattices}.

We introduce some notation and terminology. If $A$ is a partial
Boolean algebra, then a subset $T$ of pairwise commeasurable elements
which is closed under all the algebraic operations of $A$ will be
called a \emph{commeasurable} or \emph{total} subalgebra. Clearly, a
commeasurable subalgebra has the structure of a Boolean algebra. Note
that if $A$ is a partial Boolean algebra and $S$ is subset of pairwise
commeasurable elements, then there must be a \emph{smallest}
commeasurable subalgebra $T$ that contains $S$: it has to consist of
the values of Boolean expressions built from elements of $S$. We
denote it by $\generated{A}{S}$. 

Given a partial Boolean algebra $A$, the collection of its commeasurable
subalgebras $\cC(A)$ is partially ordered by inclusion. In
fact, $\cC$ is a functor $\Cat{PBoolean} \to \Cat{POrder}$ to the category
of posets and monotone functions. Regarding posets as categories,
$\cC(A)$ gives a diagram in the category $\Cat{PBoolean}$ (in fact, it
also defines a diagram in the category $\Cat{Boolean}$ of Boolean
algebras). The following proposition lists some easy properties of
this diagram. 

\begin{proposition}
\label{prop:structureCboolean}
  Let $A$ be a partial Boolean algebra.
  \begin{enumerate}
    \item[(a)] The least element of the poset $\cC(A)$ is
      $\generated{A}{0} = \generated{A}{1} = \{ 0,1\}$.
    \item[(b)] The atoms of $\cC(A)$ are
      $\generated{A}{a}=\{0,a,\neg a,1\}$ for nontrivial $a \in A$ (an
      element $p$ of a poset with least element $0$ is an atom when
      there are no elements $x$ such that $0 < x < p$). 
    \item[(c)] Two total subalgebras $S$ and $T$ have a common upper
      bound in $\cC(A)$ if and only if all elements of $S$ are
      commeasurable with all the elements of $T$.   
    \item[(d)] $A$ is a (total) Boolean algebra if and only if the poset $\cC(A)$
      is filtered (meaning that any two elements have an upper
      bound). In that case, $A$ is the largest element of the poset
      $\cC(A)$. 
  \end{enumerate}
  \end{proposition}
\begin{proof}  
  Parts (a) and (b) are easy to show and therefore we omit their proofs. 
  
  To see (c), observe that if total subalgebras $S$
  and $T$ have a common upper bound $U$, then all elements of $S$ are
  commeasurable with all elements in $T$, because all elements of $S$
  and $T$ belong to the commeasurable subalgebra $U$. Conversely, if
  all elements of $S$ are commeasurable 
  with those of $T$, then $\generated{A}{S \cup T}$ is an upper bound
  (in fact, the least upper bound) in $\cC(A)$ of $S$ and $T$.  

  If $A$ is total, then $A$ is the top element of $\cC(A)$ and hence
  $\cC(A)$ is filtered. If, on the other hand, $\cC(A)$ is filtered,
  then for any two elements $a, b \in A$ the total subalgebras
  $\generated{A}{a}$ and $\generated{A}{b}$ must have an upper bound,
  which implies (by (c)) that $a$ and $b$ are commeasurable. This shows
  (d). 
  \qed
\end{proof}

\begin{remark}
 One can show that $\cC(A)$ is a directed complete partial order
  (dcpo), which is algebraic, and is such that for every compact
  element $x$ the downset $\mathop{\downarrow}x$ is dually isomorphic to a
  finite partition lattice. The main result of the
  paper~\cite{graetzerkohmakkai:booleansubalgebras} suggests that 
  every such dcpo is the $\cC(A)$ of a unique partial Boolean algebra
  $A$. Whether similar results hold for partial
  C*-algebras remains to be seen. 
\end{remark}

We are now ready to prove the first version of our main result.

\begin{theorem}
\label{thm:colimbool}
  Every partial Boolean algebra is a colimit of its (finitely
  generated) total subalgebras.
\end{theorem}
\begin{proof}
  Let $A$ be a partial Boolean algebra, and consider its diagram of
  (finitely generated) commeasurable subalgebras $C$.
  Define functions $i_C \colon C \to A$ by the inclusions; these are
  morphisms of $\Cat{PBoolean}$ by construction. Moreover, they form a
  cocone; we will prove that this cocone is universal.
  If $f_C \colon C \to B$ is another cocone, define a function $m
  \colon A \to B$ by $m(a) = f_{\generated{A}{a}}(a)$. It now follows from
  the assumption that the $f_C$ are morphisms of $\Cat{PBoolean}$ that
  $m$ is a well-defined morphism, too. To see this, we need to show that $m$ preserves commeasurability and the algebraic operations of $A$. We check that $m$ preserves commeasurability, omitting a very similar verification that it also preserves the algebraic operations. So, if
  $a \commeas b$, then also $m(a)\commeas m(b)$, since
  $\generated{A}{a, b}$ is a total subalgebra and
  \begin{align*}
        m(a)
    & = f_{\generated{A}{a}}(a)
      = f_{\generated{A}{a, b}}(a), \\
        m(b)
    & = f_{\generated{A}{b}}(b)
      = f_{\generated{A}{a, b}}(b),
  \end{align*}
  because the $f_C$ form a cocone. One easily verifies that $f_C = m
  \after i_C$, and that $m$ is the unique such morphism.
  \qed
\end{proof}

Kalmbach's ``Bundle Lemma''~\cite{kalmbach:orthomodularlattices} gives
sufficient conditions for a family of Boolean algebras to combine
into a partial Boolean algebra, so that it could be regarded as a converse
of the previous theorem.

Notice that the morphisms of $\Cat{PBoolean}$ are the
weakest ones for which the previous theorem holds. For example, even when $A$ and $B$ are orthomodular lattices, the mediating morphism $m: A \to B$ in the proof of the previous theorem need
not be a homomorphism of orthomodular lattices. For a counterexample,
consider the function
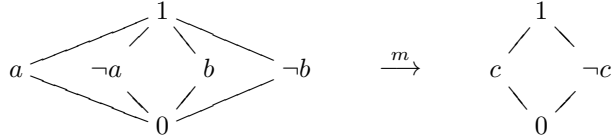
\[\vcenter{\xymatrix@C-4ex@R-4ex{
    &&& 1 \ar@{-}[dlll] \ar@{-}[dl] \ar@{-}[dr] \ar@{-}[drrr] \\
    a\llap{\phantom{$a^\perp_\perp$}}
    && \lnot a\llap{\phantom{$a^\perp_\perp$}}
    && b\rlap{\phantom{$a^\perp_\perp$}}
    && \lnot b \rlap{\phantom{$a^\perp_\perp$}}  \\
    &&& 0 \ar@{-}[ulll] \ar@{-}[ul] \ar@{-}[ur] \ar@{-}[urrr]
  }}
  \qquad\stackrel{m}{\longrightarrow}\qquad
  \vcenter{\xymatrix@C-4ex@R-4ex{
    & 1 \ar@{-}[dl] \ar@{-}[dr] \\
    c\llap{\phantom{$a^\perp_\perp$}}
    && \lnot c \rlap{\phantom{$a^\perp_\perp$}}  \\
    & 0 \ar@{-}[ul] \ar@{-}[ur]
  }}
\]
given by $m(0)=0, m(a)=m(b)=c, m(\lnot a)=m(\lnot b)=\lnot c,
m(1)=1$. It preserves $0$, $1$, $\lnot$ and $\leq$. The only commeasurable
subalgebras of the domain $A$ are $\generated{A}{0}$, $\generated{A}{a}$ and
$\generated{A}{b}$, and $m$ preserves $\land$ when restricted to
those. However, $m(a \land b) = m(0) = 0 \neq c = c \land c =
m(a) \land m(b)$. (Of course, $\{ 0, a, b, 1\}$ is a Boolean algebra,
but it is not a commeasurable subalgebra, as it does not have the same
negation $\lnot$ as $A$; see also footnote~\ref{footnote:subalgebrastructural}.) 

It follows from the previous theorem that the partial Boolean algebra
with a prescribed poset of total subalgebras is unique up to
isomorphism (of partial Boolean algebras). To actually reconstruct a
partial Boolean algebra from its total subalgebras, this should be
complemented by a description of colimits in the category
$\Cat{PBoolean}$, as we now discuss.

The coproduct of a family $A_i$ of partial Boolean algebras is got
by taking their disjoint union, identifying all the elements $0_i$,
and identifying all the elements $1_i$. Notice that elements from
different summands $A_i$ are never commeasurable in the coproduct.
In particular, the initial object $\nul$ is the partial
Boolean algebra $\{0,1\}$ with two distinct elements.

% The coequalizer of $f,g \colon A \to B$ is computed as follows. First
% take the coequalizer in the category of sets, whose elements will be
% equivalence classes of $B$, and declare two equivalence classes to be
% commeasurable when they have commeasurable representatives. Next, take
% the free partial Boolean algebra: it consists of polynomials built up
% from pairwise commeasurable equivalence classes, subject to
% identifications making them into a partial Boolean algebra, where
% polynomials are commeasurable when all pairs of their variables
% are. Finally, quotient by the smallest congruence that makes the
% canonical map from $B$ into a morphism of partial Boolean algebras.

% A congruence of partial Boolean algebras is an equivalence
% relation $\sim$ that respects the algebraic structure whenever
% defined: for example, if $a \commeas b$, $a\sim a'$ and $b \sim b'$,
% then ($a' \commeas b'$ and) $a \land b \sim a' \land b'$. Hence a
% quotient of a partial Boolean algebra by a congruence is again a partial
% Boolean algebra, and there is a canonical morphism to the quotient.
% The coequalizer of $f,g \colon A \to B$ is computed by quotienting $B$
% by the smallest congruence containing the pairs $(f(a), g(a))$ for all
% $a \in A$.

% All in all, this proves the following theorem.

Incidentally, $\Cat{PBoolean}$ is complete. Products and
equalizers of partial Boolean algebras are constructed as in the
category of sets; products have commeasurability and algebraic
structure defined componentwise, and equalizers have subalgebra
structure.  Hence the limit of a diagram of Boolean algebras is the
same in the categories of Boolean algebras and of partial Boolean
algebras. In particular, the terminal object $\een$ is the partial
Boolean algebra with a single element $0=1$.

Coequalizers are harder to describe constructively, but the following
theorem proves they do exist.

\begin{theorem}
\label{thm:PBoolcocomplete}
  The category $\Cat{PBoolean}$ is complete and cocomplete.
\end{theorem}
\begin{proof}
  We are to show that $\Cat{PBoolean}$ has coequalizers, \ie that the
  diagonal functor $\Delta \colon \Cat{PBoolean} \to
  \Cat{PBoolean}^{(\bullet \rightrightarrows \bullet)}$ has a left adjoint
  (where $\bullet \rightrightarrows \bullet$ is the free category generated by the
  graph consisting of two vertices and two parallel arrows between
  them). Since we 
  already know that $\Cat{PBoolean}$ is complete and $\Delta$
  preserves limits, Freyd's adjoint functor theorem shows that it
  suffices if the following solution set condition is
  satisfied~\cite[V.6]{maclane:categorieswork}. For
  each $f,g \colon A \to B$ in $\Cat{PBoolean}$ there is a set-indexed
  family $h_i \colon B \to Q_i$ such that $h_i f = h_i g$, and if $hf
  = hg$ then $h$ factorizes through some $h_i$.

  Take the collection of $h_i \colon B \to Q_i$ to comprise all
  `quotients', \ie (isomorphism classes of) surjections $h_i$
  of partial Boolean algebras such that $h_i f = h_i g$.
  This collection is in fact a set. The proof is finished by
  observing that every morphism $h \colon B \to Q$ of partial Boolean
  algebras factors through (a surjection onto) its set-theoretical
  image, which is a partial Boolean subalgebra of $Q$, inheriting
  commeasurability from $B$ and algebraic operations from $Q$.
  \qed
\end{proof}

\subsection{Variations}

Results similar to those above hold for many classes of Boolean
algebras, such as complete or countably complete Boolean
algebras. For example, the former variation can be defined as
follows.

\begin{definition}
  A \emph{partial complete Boolean algebra} consists of a partial
  Boolean algebra together with a (partial) operation
  \[
    \bigvee \colon \{ X \subseteq A \mid X \times X \subseteq \commeas \}
    \to A
  \]
  such that every set $S \subseteq A$ of pairwise commeasurable
  elements is contained in a set $T \subseteq A$, whose elements are
  also pairwise commeasurable, and on which the above operations
  determine a complete Boolean algebra structure.\footnote{So $T$ is
  not only closed under $\lnot, \land$ and $\lor$, but also under $\bigvee$.}

  A morphism of partial complete Boolean algebras is a function that
  preserves commeasurability and all the algebraic structure,
  including $\bigvee$, whenever defined. We write $\Cat{PCBoolean}$
  for the resulting category.
\end{definition}

A version of our main theorem also holds for such partial complete
Boolean algebras, when we define a total subalgebra of a partial
complete Boolean algebra to be a total subalgebra of the underlying
partial Boolean algebra that is additionally closed under $\bigvee$.

\begin{theorem}
  Every partial complete Boolean algebra is a colimit of its total subalgebras.
\end{theorem}
\begin{proof}
  Completely analogous to Theorem \ref{thm:colimbool}.
  \qed
\end{proof}

\section{Stone duality}
\label{sec:stone}

The full subcategory of $\Cat{PBoolean}$ consisting of (total) Boolean
algebras is just the category $\Cat{Boolean}$ of Boolean
algebras and their homomorphisms. This category is dual to the category of
Stone spaces and continuous functions via Stone
duality~\cite{johnstone:stonespaces}:
\begin{equation}
\label{eq:stone}\xymatrix{
  \Cat{Boolean} \ar@<1ex>^-{\Sigma}[rr] \ar@{}|-{\sim}[rr]
  && \Cat{Stone}\op, \ar@<1ex>^-{\Cat{Loc}(-,\{0,1\})}[ll]
}\end{equation}
where $\Sigma(A)$ is the Stone spectrum of a Boolean algebra $A$.
The dualizing object $\{0,1\}$ is both a locale and a (partial) Boolean
algebra; recall that it is in fact the initial partial Boolean algebra
$\nul$.

One might expect that the category of partial Boolean algebras enters
Stone duality~\eqref{eq:stone}, and indeed the colimit theorem,
Theorem~\ref{thm:colimbool}, enables us to prove the following extension.

\begin{proposition}
  There is a reflection
  \[\xymatrix{
    \Cat{PBoolean} \ar@<1ex>^-{K}[rr] \ar@{}|-{\perp}[rr]
    && \Cat{Stone}\op, \ar@<1ex>^-{\Cat{Loc}(-,\{0,1\})}[ll]
  }\]
  in which the functor $K$ is determined by $ K(A) = \lim_{C \in \cC(A)\op} \Sigma(C)$.
\end{proposition}
\begin{proof}
  Let $A$ be a partial Boolean algebra and $X$ a Stone space.
  Then there are bijective correspondences:
  \[\begin{bijectivecorrespondence}
    \correspondence[(in \Cat{Stone}\op)]{f \colon K(A) = \lim_{C \in
        \cC(A)\op} \Sigma(C) \to X}
    \correspondence[(in \Cat{Stone}\op)]{\forall_{C \in \cC(A)}.\;
        f_C \colon \Sigma(C) \to X}
    \correspondence[(in \Cat{Boolean})]{\forall_{C \in \cC(A)}.\;
        g_C \colon C \to \Cat{Loc}(X,\{0,1\})}
    \correspondence[(in \Cat{PBoolean}).]{g \colon A \to
      \Cat{Loc}(X,\{0,1\})}
  \end{bijectivecorrespondence}\]
  The first correspondence holds by definition of limit, the middle
  correspondence holds by Stone duality~\eqref{eq:stone}, and the
  last correspondence holds by Theorem~\ref{thm:colimbool}.
  Since all correspondences are natural in $A$ and $X$, this
  establishes the adjunction $K \dashv \Cat{Loc}(-,\{0,1\})$.
  Finally, since a Boolean algebra is trivially a colimit of
  itself in $\Cat{PBoolean}$, the adjunction is a reflection.
  \qed
\end{proof}

\begin{theorem}
  The reflection $K \dashv \Cat{Loc}(-,\{0,1\})$ extends Stone
  duality, \ie the following diagram commutes (serially).
  
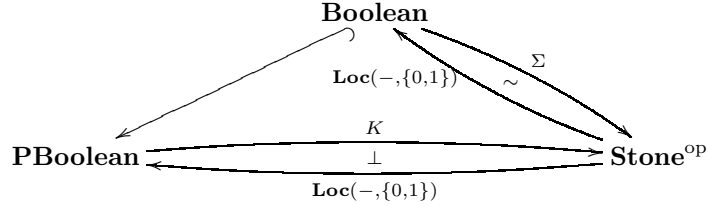
\[\xymatrix@C+9ex@R+4ex{
    & \Cat{Boolean} \ar@{^{(}->}[dl] \ar@/^/^-{\Sigma}[dr]
    \ar@{}|-{\sim}[dr] \\
    \Cat{PBoolean} \ar@/^/^-{K}[rr] \ar@{}|-{\perp}[rr]
    && \Cat{Stone}\op \ar@/^/^-{\Cat{Loc}(-,\{0,1\})}[ll]
    \ar@/^/^(.75){\Cat{Loc}(-,\{0,1\})}[ul]
  }\]
\end{theorem}
\begin{proof}
  If $A$ is a Boolean algebra, it is the initial element in the
  diagram $\cC(A)\op$ by Proposition~\ref{prop:structureCboolean}(b).
  Hence $K(A) = \lim_{C \in \cC(A)\op} \Sigma(C) = \Sigma(A)$.
  \qed
\end{proof}

\begin{corollary}
  Boolean algebras form a reflective full subcategory of the category of
  partial Boolean algebras, \ie the inclusion $\Cat{Boolean} \hookrightarrow
  \Cat{PBoolean}$ has a left adjoint $L \colon \Cat{PBoolean} \to
  \Cat{Boolean}$.
\end{corollary}
\begin{proof}
  The adjunctions of the previous theorem compose, giving the required
  left adjoint as $L=\Cat{Loc}(-,\{0,1\}) \after K$.
  \qed
\end{proof}

\section{Partial C*-algebras}
\label{sec:cstar}

The definitions of partial C*-algebras and their morphisms closely
resemble those of partial Boolean algebras. Indeed, both are
instances of the partial algebras of Kochen and Specker, over the
fields $\field{Z}_2$ and $\field{C}$, respectively. However, partial
C*-algebras also have to account for a norm and involution, calling
for some changes that we now spell out.

\begin{definition}
\label{def:pcstar}
  A \emph{partial C*-algebra} consists of a set $A$ with
  \begin{itemize}
    \item a reflexive and symmetric binary
      (\emph{commeasurability}) relation $\commeas \subseteq A \times A$;
    \item elements $0,1 \in A$;
    \item a (total) involution $* \colon A \to A$;
    \item a (total) function $\cdot \colon \field{C} \times A \to A$;
    \item a (total) function $\norm{-}: A \to \field{R}$;
    \item (partial) binary operations $+, \cdot \colon  \commeas \to A$;
  \end{itemize}
  such that every set $S \subseteq A$ of pairwise commeasurable
  elements is contained in a set $T \subseteq A$, whose elements are
  also pairwise commeasurable, and on which the above operations
  determine the structure of a commutative C*-algebra.\footnote{This
  entails that $T$ contains 0 and 1, is closed under all 
  algebraic operations, and is norm-complete.} 
\end{definition}

It follows from the last condition in the definition of a partial
C*-algebra that commeasurable elements in a partial C*-algebra have to
commute. In fact, as with partial Boolean algebras, partial
C*-algebras whose commeasurability relation is total are nothing but
\emph{commutative} C*-algebras.

We again define the notion of a \emph{commeasurable} or \emph{commutative} subalgebra of a partial C*-algebra $A$ in the obvious way as a subset $T$ of $A$ of pairwise commeasurable elements on which the operations of $A$ determine a commutative C*-algebra structure. Also, if $S$ is subset of pairwise commeasurable elements, then again there must be a \emph{smallest} commeasurable subalgebra $T$ that contains $S$: simply take the intersection of all such subalgebras $T$. Alternatively, one can construct it as the set of those elements in $A$ that are limits of sequences whose terms are algebraic expressions involving the elements of $S$. We denote it by $\generated{A}{S}$.

The reader may be tempted to believe that every noncommutative
C*-algebra can be regarded as a partial C*-algebra by declaring that
$a \commeas b$ holds whenever $a$ and $b$ commute, but that would be
incorrect. The reason for this is that we have required $\commeas$
to be reflexive, so that $aa^* = a^*a$ holds for every element $a$ in a partial C*-algebra. Now, an element $a$ such that $aa^* = a^*a$ holds is called \emph{normal} and it is not the case that every element in a C*-algebra is normal.

What is true, however, is that one may regard the collection of normal
elements of a C*-algebra as a partial C*-algebra by declaring two
elements to be commeasurable whenever they commute. In fact, taking
normal elements is part of a functor, if we consider the
following class of morphisms of partial C*-algebras.

\begin{definition}
  A \emph{partial *-morphism} is a (total) function $f \colon A \to B$
  between partial C*-algebras such that:
  \begin{itemize}
    \item $f(a) \commeas f(b)$ for commeasurable $a,b \in A$;
    \item $f(ab)=f(a)f(b)$ for commeasurable $a,b \in A$;
    \item $f(a+b)=f(a)+f(b)$ for commeasurable $a,b \in A$;
    % \item $f(a+ib)=f(a)+if(b)$ for self-adjoint $a,b \in A$;
    \item $f(za) = zf(a)$ for $z \in \field{C}$ and $a \in A$;
    \item $f(a)^* = f(a^*)$ for $a \in A$;
    \item $f(1)=1$.
  \end{itemize}
  Partial C*-algebras and partial *-morphisms organize themselves into
  a category denoted by $\Cat{PCstar}$.\footnote{Most results hold for
  nonunital C*-algebras, but for convenience we consider unital ones.}
\end{definition}

Before we embark on proving that taking normal parts provides a
functor from the category of C*-algebras to the category of partial
C*-algebras, recall that an element $a$ of a C*-algebra is called
self-adjoint when $a=a^*$, and that any element can be written
uniquely as a linear combination $a=a_1+ia_2$ of two self-adjoint
elements $a_1=\frac{1}{2}(a+a^*)$ and $a_2=\frac{1}{2i}(a-a^*)$.

\begin{proposition}
\label{prop:functorN}
  There is a functor $N \colon \Cat{Cstar} \to \Cat{PCstar}$ which
  sends every C*-algebra to its normal part
  \[
    N(A) = \{ a \in A \mid aa^* = a^*a \},
  \]
  which can be considered as a partial C*-algebra by saying that $a
  \commeas b$ holds whenever $a$ and $b$ commute. Moreover, $N$ is
  faithful and reflects isomorphisms and identities.
\end{proposition}
\begin{proof}
  The action of $N$ is well-defined on objects, since a subalgebra of a
  C*-algebra generated by a set $S$ is commutative iff the elements of
  $S$ are normal and commute pairwise. % This is probably well-known,
  % but I cannot find a reference.

  On morphisms, $N$ acts by restriction and corestriction. To see that
  it has the properties the proposition claims it has, one uses the
  identity $a=a_1+ia_2$. For example, suppose $N(f)$ is surjective for
  a *-morphism $f \colon A \to B$ and let $b \in B$. Then there are
  $a_1,a_2 \in N(A)$ with $f(a_i)=b_i$. Hence $f(a_1+ia_2)=b$, so that
  $f$ is surjective. Similarly, if $N(f)$ is injective, suppose that
  $f(a)=f(a')$. Then $f(a_i)=f(a_i')$, and hence $a_i=a_i'$, so that
  $a=a'$ and $f$ is injective. Now, isomorphisms in
  ($\Cat{P}$)$\Cat{Cstar}$ are bijective (partial) *-morphisms.
  Hence $f$ is an isomorphism when $N(f)$ is.
\qed
\end{proof}

So therefore one way of thinking about a partial C*-algebra is as
axiomatizing the normal part of a C*-algebra. Of course, we could have
decided to drop the requirement that $\commeas$ is reflexive, so that every C*-algebra would also be a partial C*-algebra,
with commeasurability given by commutativity. We haven't done this for
various reasons. First of all, we would like to have a notion given in
terms of the physically relevant data. On Bohr's philosophy, which we
adopt in this paper, the physically relevant information is contained
in the normal part of a C*-algebra. Related to this is the fact that
the Bohrification functor (which we will study in
Section~\ref{sec:bohrification}) only takes the
normal part of a C*-algebra into account. Secondly, we wish to have a
result saying how a partial C*-algebra is determined by its
commutative subalgebras analogous to our result for partial Boolean
algebras. With our present definition we do indeed have this result
(it is Theorem \ref{thm:colimcstar} below), as we will now explain.

Denote by $\cC \colon \Cat{PCstar} \to \Cat{POrder}$ the functor
assigning to a partial C*-algebra $A$ the collection of its
commeasurable (\ie commutative total) subalgebras $\cC(A)$, partially
ordered by inclusion. One immediately derives similar properties for
the diagram $\cC(A)$ in the category $\Cat{PCstar}$ as for partial
Boolean algebras.

\begin{proposition}
\label{prop:structureCcstar}
  Let $A$ be a partial C*-algebra.
  \begin{enumerate}
    \item[(a)] The least element of the poset $\cC(A)$ is
      $\generated{A}{0} = \generated{A}{1} = \{ z \cdot 1 \mid
      z \in \field{C} \}$.
    \item[(b)] The poset $\cC(A)$ is filtered if and only if $A$
      is a commutative C*-algebra. \\ In that case, $A$ is the largest
      element of the poset $\cC(A)$.
  \end{enumerate}
  \qed
\end{proposition}

We now prove the C*-algebra version of our main result.

\begin{theorem}
\label{thm:colimcstar}
  Every partial C*-algebra is a colimit of its (finitely generated)
  commutative C*-subalgebras.
\end{theorem}
\begin{proof}
  Let $A$ be a partial C*-algebra, and consider its diagram $\cC(A)$
  of (finitely generated) commutative C*-subalgebras $C$. Defining
  functions $i_C \colon C
  \to A$ by the inclusions yields a cocone in $\Cat{PCstar}$; we will
  prove that this cocone is universal. If $f_C \colon C \to B$ is
  another cocone, define a function $m \colon A \to B$ by
  $
    m(a) = f_{\generated{A}{a}}(a)
  $.
  Precisely as in Theorem~\ref{thm:colimbool}, it now follows from the
  assumption that the $f_C$ are morphisms of $\Cat{PCstar}$ that $m$
  is a well-defined morphism, too.  
% For example, if
%   $a \commeas b$, then also $m(a)\commeas m(b)$, since
%   $\generated{A}{a, b}$ is a commutative C*-subalgebra and
%   \begin{align*}
%         m(a)
%     & = f_{\generated{A}{a}}(a)
%       = f_{\generated{A}{a, b}}(a), \\
%         m(b)
%     & = f_{\generated{A}{b}}(b)
%       = f_{\generated{A}{a, b}}(b),
%   \end{align*}
%   because the $f_C$ form a cocone. 
  One easily verifies that $f_C = m
  \after i_C$, and that $m$ is the unique such morphism.
  \qed
  \qed
\end{proof}

Together, Theorem~\ref{thm:colimcstar} above and
Theorem~\ref{thm:pcstarcocomplete} below show
that, up to partial *-isomorphism, every C*-algebra can be
reconstructed from its commutative C*-subalgebras, lending force to the Bohrification programme. In this light
Theorem~\ref{thm:colimcstar} could be said to embody a categorical
crude version of complementarity.

\begin{remark} 
\label{remark:disclaimer}
  We hasten to point out that Theorem
  \ref{thm:colimcstar} only works because of the way we have set
  things up. In particular, it would fail if we would drop the
  requirement that $\commeas$ is reflexive. Also, it does not say that
  every C*-algebra is the colimit of its commutative subalgebras in
  the category $\Cat{Cstar}$, which would be false. 

  The reason why Theorem \ref{thm:colimcstar} cannot be changed in these
  ways is that are nonisomorphic von Neumann algebras $A$ and $B$ for
  which $\cC(A)$ and $\cC(B)$ \emph{are} isomorphic. (This follows from
  the work of Connes in \cite{connes:factornotantiisomorphic}. For the
  experts, the argument is this: in
  \cite{connes:factornotantiisomorphic} it is shown that there is a von
  Neumann algebra $A$ that is not 
  anti-isomorphic to itself; since
  this $A$ is in standard form, \ie has a separating cyclic vector, it
  follows that $A$ is not isomorphic to its commutant $A'$. But
  Tomita-Takesaki theory shows that any such von Neumann algebra $A$ is
  anti-isomorphic to its commutant $A'$, whence $\cC(A) \cong 
  \cC(A')$.) Note that this, combined with Theorem~\ref{thm:colimcstar},
  implies that there is a partial *-isomorphism $N(A) \to N(B)$ that is
  not of the form $N(f)$ for some $f \colon A \to B$. To put it another
  way, the faithful functor $N$ is not full. 

  Because of this the relevance of the current work to the theory of
  C*-algebras proper is rather limited. But that is not our immediate
  aim: this paper primarily wishes to gain a better conceptual understanding of the
  Bohrification programme; and, as we have argued
  above, from that perspective the way we have set things up is very
  natural and Theorem \ref{thm:colimcstar} is a step forward. 
\end{remark}

We close this section with a discussion of completeness and
cocompleteness properties of $\Cat{PCstar}$. It is known that the
category of C*-algebras is both complete and cocomplete (for
coproducts, see~\cite{rainjonneau:coproductscstar}, and for
coequalizers, see~\cite{guichardet:vonneumann}). As it turns out, also
$\Cat{PCstar}$ is both complete and cocomplete.

$\Cat{PCstar}$ is complete, as it has both equalizers and arbitrary
products. Equalizers of partial C*-algebras are constructed as in
$\Cat{Set}$, having inherited commeasurability and subalgebra
structure. Products are given by $\prod_i A_i = \{ (a_i)_i \mid a_i
\in A_i, \sup_i \|a_i\| < \infty \}$, with componentwise
commeasurability and algebraic structure.
% Hence the limit of a finite diagram of
% C*-algebras is the same in the category of C*-algebras and in the
% category of partial C*-algebras.
In particular, the terminal object $\een$ is the 0-dimensional
(partial) C*-algebra $\{0\}$, confusingly sometimes also denoted by 0.

The coproduct of a family $A_i$ of partial C*-algebras is got
by taking their disjoint union, identifying for every $z \in
\field{C}$ the elements of the form  $z1_i$. Notice that elements from
different summands $A_i$ are never commeasurable in the coproduct.
In particular, the initial object $\nul$ is the 1-dimensional
(partial) C*-algebra $\field{C}$, which is confusingly sometimes also
denoted by 1.

Coequalizers are harder to describe constructively, but they do exist.

\begin{theorem}
\label{thm:pcstarcocomplete}
  The category $\Cat{PCstar}$ is complete and cocomplete.
\end{theorem}
\begin{proof}
  To show that $\Cat{PCstar}$ has coequalizers, the same strategy as in
  the proof of Theorem~\ref{thm:PBoolcocomplete} applies, because for
  every partial C*-algebra $A$ the collection of isomorphism classes
  of partial *-maps $f: A \to B$ such that $f(B)$ is dense in $A$ form
  a set and every partial C*-algebra map with domain $A$ factors
  through a map of this form.
  \qed
\end{proof}

\subsection{Variations}

Theorem~\ref{thm:colimcstar} holds for many varieties of (partial)
C*-algebras, as its proof only depends on (partial) algebraic
properties. Let us consider (partial) Rickart C*-algebras as an
example. Recall that a commutative C*-algebra $A$ is Rickart
when every $a \in A$ has a unique projection $\RP(a) \in A$ such that
$(1-\RP(a)) \cdot A$ is the right annihilator $\{ b \in A \mid ab = 0
\}$. We call a partial C*-algebra $A$ together with a total map
\[
  \RP \colon A \to A
\]
a \emph{partial Rickart C*-algebra} when every pairwise commeasurable
$S \subseteq A$ is contained in a pairwise commeasurable $T \subseteq
A$ on which the operations of $A$ yield a commutative Rickart
C*-algebra structure with RP's given by the function above.
Denote the subcategory of $\Cat{PCstar}$ whose objects
are partial Rickart C*-algebras and whose morphism are partial
*-morphisms that preserve RP by $\Cat{PRickart}$.

\begin{theorem}
  Every partial Rickart C*-algebra is the colimit of its
  commutative Rickart C*-subalgebras.
\end{theorem}
\begin{proof}
  The proof of Theorem~\ref{thm:colimcstar} holds verbatim when every
  reference to (partial) C*-algebras is replaced by (partial)
  Rickart C*-algebras.
  \qed
\end{proof}

If $\Cat{Rickart}$ is the subcategory of $\Cat{Cstar}$ consisting of
Rickart C*-algebras and *-morphisms preserving RP, then there is a
functor
\[
  N: \Cat{Rickart} \to \Cat{PRickart}
\]
sending every Rickart C*-algebra $A$ to its normal part; this follows
from \cite[Proposition~4.4]{berberian:baeronlinebook}. Similar
results hold for any type of C*-algebra that is defined by algebraic
properties, such as AW*-algebras and spectral C*-algebras (see
\cite[5.1]{heunenlandsmanspitters:bohrification}). We will come back
to AW*-algebras in Section~\ref{sec:projections} below.

The distinguishing feature of von Neumann algebras amongst
C*-algebras, in contrast, is topological in nature. This makes it
harder to come up with a notion of partial von
Neumann algebra: the obvious definition -- a partial
C*-algebra $A$ in which every subset of commeasurable elements is contained in a von Neumann algebra --
has the drawback that it is not clear if $N(A)$ would be a partial von
Neumann algebra given a von Neumann algebra $A$. We can, however,
still obtain the following.

Let $A$ be a von Neumann algebra. Without loss of generality, we may
assume that $A$ acts on a Hilbert space $H$. Denote the von Neumann
subalgebra of a von Neumann algebra $A$ generated by a subset $S
\subseteq A$ by $\Generated{A}{S}$. It is the closure of the
C*-algebra $\generated{A}{S}$ in the weak operator
topology, and by von Neumann's double commutant theorem~\cite[Theorem
5.3.1]{kadisonringrose:operatoralgebras}, it equals $\generated{A}{S}''$.

\begin{lemma}
\label{lem:VNgencomm}
  If a C*-subalgebra $C$ of a von Neumann algebra $A$ is commutative,
  then so is its von Neumann envelope $\Generated{A}{C}$.
  Hence if $a \in A$ is normal, then $\Generated{A}{a}$ is commutative.
\end{lemma}
\begin{proof}
  Let $a,b \in C''$. Since $C''$ is the (weak operator) closure of
  $C$, we can write $b$ as a (weak operator) limit
  $b=\lim_n b_n$ for $b_n \in C$. Then:
  \begin{align*}
        ab
      = a(\lim_n b_n) 
    & = \lim_n ab_n
        \eqcomment{(by \cite[5.7.9(i)]{kadisonringrose:operatoralgebras})} \\
    & = \lim_n b_na
        \eqcomment{(since $a \in C''$ and $b_n \in C \subseteq C'$)} \\
    & = (\lim_n b_n)a
        \eqcomment{(by \cite[5.7.9(ii)]{kadisonringrose:operatoralgebras})} \\
    & = ba.
  \end{align*}
  \qed
\end{proof}

\begin{theorem}
\label{thm:colimvonneumann}
  Let $A$ be a von Neumann algebra acting on a Hilbert space.
  Then $N(A)$ is a colimit in $\Cat{PCstar}$ of the
  (finitely generated) commutative von Neumann subalgebras of $A$.
\end{theorem}
\begin{proof}
  Using Lemma~\ref{lem:VNgencomm}, the proof of
  Theorem~\ref{thm:colimcstar} holds verbatim when every
  occurence of $\generated{A}{S}$ is replaced by $\Generated{A}{S}$.
  \qed
\end{proof}

\section{Gelfand duality}
\label{sec:gelfand}

The full subcategory of $\Cat{PCstar}$ consisting of commutative
C*-algebras is just the category $\Cat{cCstar}$ of commutative
C*-algebras and *-morphisms. This category is dual to the category of
compact Hausdorff spaces and continuous functions via Gelfand
duality~\cite{johnstone:stonespaces}. Constructively, the latter
category is replaced by that of compact completely regular
locales~\cite{coquandspitters:gelfand}:
\begin{equation}
\label{eq:gelfand}\xymatrix{
  \Cat{cCstar} \ar@<1ex>^-{\Sigma}[rr] \ar@{}|-{\sim}[rr]
  && \Cat{KRegLoc}\op, \ar@<1ex>^-{\Cat{Loc}(-,\field{C})}[ll]
}\end{equation}
where $\Sigma(A)$ is the Gelfand spectrum of a commutative C*-algebra
$A$. The dualizing object $\field{C}$ is both a locale and a (partial)
C*-algebra; recall that it is in fact the initial partial C*-algebra
$\nul$.

The colimit theorem, Theorem~\ref{thm:colimcstar}, together with the
fact that the categories in~\eqref{eq:gelfand} are cocomplete and
complete, enables us to prove the following extension of Gelfand duality.

\begin{proposition}
  There is a reflection
  \[\xymatrix{
    \Cat{PCstar} \ar@<1ex>^-{K}[rr] \ar@{}|-{\perp}[rr]
    && \Cat{KRegLoc}\op, \ar@<1ex>^-{\Cat{Loc}(-,\field{C})}[ll]
  }\]
  in which the functor $K$ is determined by $ K(A) = \lim_{C \in \cC(A)\op} \Sigma(C)$.
\end{proposition}
\begin{proof}
  Let $A$ be a partial C*-algebra and $X$ a compact completely
  regular locale. Then there are bijective correspondences:
  \[\begin{bijectivecorrespondence}
    \correspondence[(in \Cat{KRegLoc}\op)]{f \colon K(A) = \lim_{C \in
        \cC(A)\op} \Sigma(C) \to X}
    \correspondence[(in \Cat{KRegLoc}\op)]{\forall_{C \in \cC(A)}.\;
        f_C \colon \Sigma(C) \to X}
    \correspondence[(in \Cat{cCstar})]{\forall_{C \in \cC(A)}.\;
        g_C \colon C \to \Cat{Loc}(X,\field{C})}
    \correspondence[(in \Cat{PCstar}).]{g \colon A \to
      \Cat{Loc}(X,\field{C})}
  \end{bijectivecorrespondence}\]
  The first correspondence holds by definition of limit, the middle
  correspondence holds by Gelfand duality~\eqref{eq:gelfand}, and the
  last correspondence holds by Theorem~\ref{thm:colimcstar}.
  Since all correspondences are natural in $A$ and $X$, this
  establishes the adjunction $K \dashv \Cat{Loc}(-,\field{C})$.
  Finally, since a commutative C*-algebra is trivially a colimit of
  itself in $\Cat{PCstar}$, the adjunction is a reflection.
  \qed
\end{proof}

\begin{theorem}
\label{thm:reflection}
  The reflection $K \dashv \Cat{Loc}(-,\field{C})$ extends Gelfand
  duality, \ie the following diagram commutes (serially).
  \[\xymatrix@C+7ex@R+3ex{
    & \Cat{cCstar} \ar@{^{(}->}[dl] \ar@/^/^-{\Sigma}[dr]
    \ar@{}|-{\sim}[dr] \\
    \Cat{PCstar} \ar@/^/^-{K}[rr] \ar@{}|-{\perp}[rr]
    && \Cat{KRegLoc}\op \ar@/^/^-{\Cat{Loc}(-,\field{C})}[ll]
    \ar@/^/^-{\Cat{Loc}(-,\field{C})}[ul]
  }\]
\end{theorem}
\begin{proof}
  If $A$ is a commutative C*-algebra, it is the initial
  element in the diagram $\cC(A)\op$ by
  Proposition~\ref{prop:structureCcstar}(b). Hence $K(A) = \lim_{C \in
    \cC(A)\op} \Sigma(C) = \Sigma(A)$.
  \qed
\end{proof}

\begin{corollary}
  Commutative C*-algebras form a reflective full subcategory of
  partial C*-algebras, \ie the inclusion $\Cat{cCstar} \hookrightarrow
  \Cat{PCstar}$ has a left adjoint $L \colon \Cat{PCstar} \to \Cat{cCstar}$.
\end{corollary}
\begin{proof}
  The adjunctions of the previous theorem compose, giving the required
  left adjoint as $L=\Cat{Loc}(-,\field{C}) \after K$.
  \qed
\end{proof}

This means that for a partial C*-algebra $A$ one has
\[
  \Cat{PCstar}(A,\field{C}) \cong \Cat{cCstar}(L(A),\field{C}).
\]
In other words, multiplicative quasi-states of $A$ that are
multiplicative on commutative subalgebras precisely correspond to
states of $L(A)$. Thus these quasi-states have good (categorical)
behaviour. 
However, things are not as interesting as they may seem.
By the Kochen-Specker theorem, no von Neumann algebra $A$ without
factors of type $I_1$ or $I_2$ can have such states~(\cite{doering:kochenspecker}, see also
\cite{redei:quantumlogic}). It follows that $K(A)=\nul$
and hence $L(A)=\een$ for such algebras. More generally, let us call a
partial C*-algebra $A$ \emph{Kochen-Specker} when $L(A)=\een$.
Any such algebra $A$ has no quasi-states:
$\Cat{PCstar}(A,\field{C}) \cong \Cat{cCstar}(\een,\field{C}) =
\emptyset$. Also, Kochen-Specker partial C*-algebras are a
`coproduct-ideal' in a sense that we now make precise.
For $X \in \Cat{KRegLoc}$ we have $\nul \times X = \nul$, so by Gelfand
duality~\eqref{eq:gelfand}, we have $\een + C = \een$ for a
commutative C*-algebra $C$. So if $A \in \Cat{PCstar}$ is
Kochen-Specker, and $B \in \Cat{PCstar}$ arbitrary, then also $A+B$ is
Kochen-Specker:
\[
  L(A+B) = L(A) + L(B) = \een + L(B) = \een.
\]
The first equality holds because $L$, being a left adjoint, preserves
colimits.
Nevertheless, the reflection of Theorem~\ref{thm:reflection} is still
interesting. Even though it does not teach much about the theory of
C*-algebras proper, it is an important step in seeing how far $A$ can
be reconstructed from $\cC(A)$ (or its Bohrification $\underline{A}$, see
Section~\ref{sec:bohrification}). See also Remark~\ref{remark:disclaimer}.

\section{Projections, partial AW*-algebras and tensor products}
\label{sec:projections}

This section discusses a functor $\Cat{PCstar} \to \Cat{PBoolean}$,
relating Sections \ref{sec:bool} and \ref{sec:stone} to
Sections \ref{sec:cstar} and \ref{sec:gelfand}.

\subsection{Projections and partial AW*-algebras}

An element $p$ of a partial C*-algebra $A$ is called a
\emph{projection} when it satisfies $p^* = p =
p^2$. The elements $0 \in A$ and $1 \in A$ are trivially projections;
other projections are called nontrivial.

\begin{lemma}
  There is a functor $\Proj \colon \Cat{PCstar} \to \Cat{PBoolean}$
  where $\Proj(A)$ is the set of projections of $A$.
\end{lemma}
\begin{proof}
  First, $\Proj(A)$ is indeed a partial Boolean
  algebra. Commeasurability is inherited from $A$. One easily checks
  that $\neg p = 1-p$ is a projection when $p$ is. If $p,q$ are
  commeasurable in $\Proj(A)$, then they commute, whence the
  projection $p \land q = pq$ is also in
  $A$~\cite[4.14]{redei:quantumlogic}. This makes $\Proj(A)$ into a
  partial Boolean algebra.
  % See also \cite[Proposition~6.3]{redei:quantumlogic}.
  Finally, morphisms of partial C*-algebras are easily seen to
  preserve projections, making the assignment $A \mapsto \Proj(A)$
  functorial.
  \qed
\end{proof}

For the following class of partial C*-algebras we get stronger results.

\begin{definition}
  A partial Rickart C*-algebra $A$ is a \emph{partial AW*-algebra}, if
  it comes equipped with an operation
  \[
    \bigvee \colon \{ X \subseteq \Proj(A) \mid X \times X \subseteq
    \commeas \} \to \Proj(A),
  \]
  in such a way that each pairwise commeasurable $S \subseteq A$
  is contained in a pairwise commeasurable $T \subseteq A$ on which
  the operations determine a commutative AW*-algebra structure
  (\ie the structure is that of a commutative Rickart C*-algebra, whose
  projections form a complete Boolean algebra with suprema given by
  the operation $\bigvee$ above). Denote the subcategory of $\Cat{PCstar}$
  whose objects are partial AW*-algebras and whose morphisms are
  partial *-morphisms which preserve $\RP$ and $\bigvee$ by
  $\Cat{PAWstar}$.
\end{definition}

\begin{lemma}
  The functor $\Proj$ restricts to a functor $\Cat{PAWstar} \to
  \Cat{PCBoolean}$. 
  \end{lemma}
\begin{proof}
  Clear from the definition of a partial AW*-algebra.
  \qed
\end{proof}

\begin{remark}
  It would be interesting to see whether this functor is part
  of an equivalence, like in the total case, where it is one side of
  an equivalence of categories between $\Cat{cAWstar}$ and
  $\Cat{CBoolean}$. 
\end{remark}

\begin{proposition}
  The functors $\Proj$ and $\cC$ commute for partial AW*-algebras:
  writing  $\cC'$ for the functor $\Cat{PAWstar} \to
  [\Cat{POrder},\Cat{cAWstar}]$, and $\cC$ for the functor $\Cat{PBoolean}
  \to \Cat{POrder}$, we have $\cC \after \Proj = \Proj
  \after \cC'$. Explicitly,
  \[
    \{ \Proj(C) \mid C \in \cC(A) \} = \cC(\Proj(A))
  \]
  for every partial AW*-algebra $A$.
\end{proposition}
\begin{proof}
  This follows from the combination of Stone and Gelfand duality,
  which yields an equivalence between $\Cat{cAWstar}$ and
  $\Cat{CBoolean}$. One direction of the equivalence
  is obtained by taking projections and the other is obtained by
  taking $C(X)$ where $X$ is the Stone space associated to the
  (complete) Boolean algebra. For the purposes of the proof, we will
  denote the latter composite functor by $F$. In particular, the
  projections $\Proj(C)$ of a commutative AW*-algebra $C$
  form a complete Boolean algebra and the left-hand side is contained
  in the right-hand side.

  For the converse, let $B$ be a complete Boolean lattice of
  projections in $A$. Then the projections in $B$ commute pairwise,
  and hence generate a commutative AW*-subalgebra
  $C=\generated{A}{B}$. We obviously have an inclusion $i: B \subseteq
  \Proj(C)$. But then the composite 
  \[\xymatrix{
      FB \ar@{ >->}[r]^(.4){Fi} 
    & F\Proj(C) \ar[r]^(.6){\eta^{-1}_C}_(.6)\cong 
    & C, 
  }\]
  where $\eta$ is the unit of the adjunction $\Proj \dashv F$, shows
  that $FB$ is isomorphic to a commutative AW*-subalgebra of $A$
  contained in $C$. This commutative subalgebra also contains $B$,
  because the diagram 
  \[\xymatrix{
       \Proj(FB) \ar@{ >->}[rr]^(.45){\Proj(Fi)}
                 \ar[d]^\cong_{\epsilon_B} 
    && \Proj(F\Proj(C)) \ar[d]_\cong^{\epsilon_{\Proj(C)} = \Proj(\eta_C^{-1})}  \\
       B \ar@{ >->}[rr]_i 
    && \Proj(C) 
  }\]
  commutes by naturality of the counit $\epsilon: \Proj (F-)
  \Rightarrow 1$. Since $C = \generated{A}{B}$, this implies that $Fi$
  is an isomorphism. But then so is $i$ and therefore $B = \Proj(C)$. 
  \qed
\end{proof}

As a corollary to the previous proposition, we can extend
Proposition \ref{prop:structureCcstar} with atoms to mirror
Proposition \ref{prop:structureCboolean}. Keep in mind that the following
corollary does not entail that $\cC(A)$ is atomic.

\begin{corollary}
  For a partial AW*-algebra $A$, the atoms of the poset $\cC(A)$ are
  $\generated{A}{p}$ for nontrivial projections $p$.
  \qed
\end{corollary}

\subsection{Tensor products}

It is clear from the description of coproducts in
$\Cat{PCstar}$ and $\Cat{PBoolean}$ that the functor $\Proj$ preserves
coproducts.
Recall that in a coproduct of partial Boolean or C*-algebras,
nontrivial elements from different summands are never commeasurable.
Theorems \ref{thm:colimbool} and \ref{thm:colimcstar} provide
the option of defining a tensor product satisfying the adverse
universal property.

\begin{definition}
\label{def:tensor}
  Let $A$ and $B$ be a pair of partial Boolean algebras (partial
  C*-algebras). Define
  \[
    A \tensor B = \colim \{ C + D \mid C \in \cC(A), D \in \cC(B) \},
  \]
  where $C+D$ is the coproduct in the category of Boolean algebras
  (commutative C*-algebras).
\end{definition}

There are canonical morphisms $\kappa_A \colon A \to A \tensor B$ and
$\kappa_B \colon B \to A \tensor B$ as follows. By definition, $A
\tensor B$ is the colimit of $C+D$ for $C \in \cC(A)$ and $D \in
\cC(B)$. Precomposing with the coproduct injections $C \to C+D$ gives
a cocone $C \to A \tensor B$ on $\cC(A)$. By the colimit theorem, $A$
is the colimit of $\cC(A)$. Hence there is a mediating morphism
$\kappa_A \colon A \to A \tensor B$.

The unit element for both the tensor product and the coproduct is the
initial object $\nul$. The big difference between $A \tensor B$ and the
coproduct $A+B$ is that elements $\kappa_A(a)$ and $\kappa_B(b)$ are
always commeasurable in the former, but never in the latter. Indeed,
this universal property characterizes the tensor product.

\begin{proposition}
  Let $f \colon A \to Z$ and $g \colon B \to Z$ be morphisms in the
  category $\Cat{PBoolean}$ ($\Cat{PCstar}$). The cotuple
  $\cotuple{f}{g} \colon A+B \to Z$ factorizes through $A \tensor B$
  if and only if $f(a) \commeas f(b)$ for all $a \in A$ and $b \in B$.
\end{proposition}
\begin{proof}
  By construction, giving $h \colon A \tensor B \to Z$ amounts to
  giving a cocone $C+D \to Z$ for $C \in \cC(A)$ and $D \in
  \cC(B)$. Because $C+D$ is totally defined, any morphism $C+D \to Z$
  must also be total. But this holds (for all $C$ and $D$) if and only
  if $f(a)$ and $g(b)$ are commeasurable for all $a \in A$ and $b \in
  B$, for (only) then can one take $h$ to be the cotuple of the
  corestrictions of $f$ and $g$.
  \qed
\end{proof}

The tensor products of Definition \ref{def:tensor} makes
$\Proj \colon \Cat{PCstar} \to \Cat{PBoolean}$ a monoidal
functor: the natural transformation $\Proj(A) \tensor
\Proj(B) \to \Proj(A \tensor B)$ is induced by the cotuples $\Proj(C)
+ \Proj(D) \to \Proj(A \tensor B)$ of
\[\xymatrix@1@C+5ex{
    \Proj(C) \ar^-{\Proj(C \hookrightarrow A)}[r]
  & \Proj(A) \ar^-{\Proj(\kappa_A)}[r]
  & \Proj(A \tensor B).
}\]

\begin{proposition}
  The functor $\Proj \colon \Cat{PCstar} \to \Cat{PBoolean}$ preserves
  coproducts and is monoidal.
  \qed
\end{proposition}

We end this section by discussing the relation between the tensor
products of partial Boolean algebras and those of Hilbert spaces, describing
compound quantum systems. Let $\Cat{Hilb}$ be the category of Hilbert
spaces and continuous linear maps, and let $B \colon \Cat{Hilb} \to
\Cat{PCstar}$ denote the functor $B(H) = \Cat{Hilb}(H,H)$ acting on
morphisms as $B(f) = f \after (-) \after f^\dag$ where $f^\dag$ is the
adjoint of $f$. The definition of 
the tensor product in $\Cat{PCstar}$ as a colimit yields a natural
transformation $B(H) \tensor B(K) \to B(H \tensor K)$, induced by
morphisms $C \to B(H \tensor K)$ for $C \in \cC(B(H))$ given by $a
\mapsto a \tensor \id[K]$. Initiality of the tensor unit $\nul$ gives
a morphism $\nul \to B(\field{C})$, and these data satisfy the coherence
requirements. Hence the functor $B$ is monoidal, and therefore also
the composite $\Proj \after B \colon \Cat{Hilb} \to \Cat{PBoolean}$ is
a monoidal functor.

\section{Functoriality of Bohrification}
\label{sec:bohrification}

The so-called Bohrification construction
(see~\cite{heunenlandsmanspitters:bohrification}, whose notation we
adopt) associates to every C*-algebra $A$ an internal commutative C*-algebra
$\underline{A}$ in the topos $[\cC(A), \Cat{Set}]$, given by the
tautological functor $\underline{A}(C)=C$. Gelfand duality then yields
an internal locale, which can in turn be externalized. As it happens this construction works equally well for partial C*-algebras, so that Bohrification for ordinary C*-algebras can be seen as the composition of the functor $N$ from Proposition \ref{prop:functorN} with Bohrification for partial C*-algebras. Thus a locale
is associated to every object of $\Cat{PCstar}$. In this final section we
consider its functorial aspects. It turns out that the whole
construction summarized above can be made into a functor from
partial C*-algebras to locales by restricting the morphisms of the former.

Bohrification does not just assign a topos to each (partial)
C*-algebra, it assigns a topos with an internal C*-algebra. To reflect
this, we define categories of toposes equipped with internal structures.

\begin{definition}
  The category $\Cat{RingedTopos}$ has as objects pairs $(T,R)$ of a
  topos $T$ and an internal ring object $R \in T$.
  A morphism $(T,R) \to (T',R')$ consists of a geometric morphism $F
  \colon T' \to T$ and an internal ring morphism $\varphi \colon R' \to
  F^*(R)$ in $T$.

  By $\Cat{CstaredTopos}$ we denote the subcategory of
  $\cat{RingedTopos}$ of objects $(T,A)$ where $A$ is an internal
  C*-algebra in $T$ and morphisms $(F,\varphi)$ where $\varphi$ is an
  internal *-ring morphism.
\end{definition}

Notice that the direction of morphisms in this definition is opposite
to the customary one in algebraic geometry~\cite[4.1]{grothendieck:ega1}.

First of all, any functor $\cat{D} \to \cat{C}$ induces a geometric
morphism $[\cat{D}, \Cat{Set}] \to [\cat{C}, \Cat{Set}]$, of which the
inverse part is given by precomposition
(see~\cite[A4.1.4]{johnstone:elephant}). We have already seen that
$\cC$ is a functor $\Cat{PCstar}\op \to \Cat{POrder}\op$. Additionally,
restricting a morphism $f \colon B \to A$ of partial C*-algebras to $D
\in \cC(B)$ and corestricting to $\cC f(D)$ gives a morphism of
commutative C*-algebras. Hence we obtain a geometric
morphism of toposes $[\cC(B), \Cat{Set}] \to [\cC(A), \Cat{Set}]$ as
well as an internal morphism of commutative *-rings. The latter is a
natural transformation whose component at $D$ is $\underline{B}(D) \to
((\cC f)^* \underline{A})(D) = \underline{A}(\cC f(D))$. In other
words, we have a functor $\Cat{PCstar}\op \to \Cat{RingedTopos}$.

In general, (inverse parts of) geometric morphisms do not preserve
internal C*-algebras. But in this particular case, $(\cC f)^*
\underline{A}$ is in fact an internal C*-algebra in $[\cC(B),
\Cat{Set}]$. The proof is contained
in~\cite[4.8]{heunenlandsmanspitters:bohrification}, which essentially
shows that any functor from a poset $P$ to the category of C*-algebras
is always an internal C*-algebra in the topos $[P, \Cat{Set}]$.
Therefore, we really have a functor $\Cat{PCstar}\op \to
\Cat{CstaredTopos}$, as the following proposition records.

\begin{proposition}\label{prop:pcstarfunctorialcstaredtopos}
  Bohrification is functorial $\Cat{PCstar}\op \to \Cat{CstaredTopos}$.
  \qed
\end{proposition}

Next, we can apply constructive Gelfand duality $\underline{\Sigma}$
internally. 

\begin{definition} 
  The category $\cat{LocaledTopos}$ has as objects pairs $(T, L)$ of a
  topos $T$ and an internal locale object $L$. A morphism $(T, L) \to
  (T', L')$ consists of a geometric morphism $F: T'\to T$ and an
  internal locale morphism $\varphi: F^*(L) \to L'$ in $T'$. 
\end{definition}

\begin{proposition}\label{prop:pcstarfunctoriallocaledtopos}
  Bohrification is functorial $\Cat{PCstar}\op \to \Cat{LocaledTopos}$.
\end{proposition}
\begin{proof}
  It is clear from the description in~\cite{coquandspitters:gelfand}
  that the construction of the generating lattice of the internal Gelfand spectrum is
  geometric. Therefore it is preserved by (inverse image parts of)
  geometric morphisms. Hence taking the internal Gelfand spectrum of a
  commutative C*-algebra commutes with inverse image functors of
  geometric morphisms. The proof is now finished by composing the
  functor of  Proposition~\ref{prop:pcstarfunctorialcstaredtopos} with the
  following: on objects, $(T,A)$ maps to $(T,\underline{\Sigma}(A))$,
  and on morphisms, $(F,\varphi)$ maps to
  $(F,\underline{\Sigma}(\varphi))$. The latter morphism
  $\underline{\Sigma}(\varphi) \colon \underline{\Sigma}(B) \to
  F^*(\underline{\Sigma}(A))$ is indeed well-defined via
  $F^*(\underline{\Sigma}(A)) \cong \underline{\Sigma}(F^*(A))$. 
  \qed
\end{proof}

However, this is where the current line of reasoning stops:
there is no clear functor $\Cat{LocaledTopos} \to
\Cat{Loc}$. %  to postcompose
% Proposition~\ref{prop:pcstarfunctoriallocaledtopos} with, such that one
% can reconstruct the Gelfand spectrum in the commutative case.
The most natural way to get ahead is to restrict the morphisms of
$\Cat{PCstar}$ as follows. 

\begin{lemma}
\label{lem:reflectcommeas}
  For morphisms $f \colon A \to B$ of $\Cat{PCstar}$, the following
  are equivalent:
  \begin{enumerate}
    \item[(a)] if $\cC f(C) \leq D$ and $\cC f(C') \leq D$ for $C,C' \in
      \cC(A)$ and $D \in \cC(B)$, then there is $C'' \in \cC(A)$ such
      that $C \leq C''$ and $C' \leq C''$ and $\cC f (C'') \leq D$;
    \item[(b)] $a \commeas a'$ when $f(a) \commeas f(a')$.
  \end{enumerate}
\end{lemma}
\begin{proof}
  First assume (a) and suppose $f(a) \commeas f(a')$. Take
  $C=\generated{A}{a,a^*}$, $C'=\generated{A}{a',(a')^*}$, and
  $D = \generated{B}{f(a), f(a'), f(a)^*, f(a')^*}$.
  Then $\cC f(C) \leq D$ and $\cC f(C') \leq D$. Hence there is $C''$
  with $C \leq C''$ and $C' \leq C''$. So $a,a'$ are both elements of
  the commutative algebra $C''$, so $a \commeas a'$.

  Conversely, assuming (b) and supposing $\cC f(C) \leq D$ and $\cC
  f(C') \leq D$, for all $a \in C$ and $a' \in C'$ we have $f(a),f(a')
  \in D$, so that $f(a) \commeas f(a')$. But that means that $C$ and
  $C'$ are commuting commutative subalgebras of $A$. Hence we can take
  $C'' = \generated{A}{C,C'}$.
  \qed
\end{proof}

We say that morphisms satisfying the conditions in the previous lemma
\emph{reflect commeasurability}. Notice that this class of morphisms
excludes the type of counterexample discussed after
Theorem~\ref{thm:colimbool}. To show how the assignment of a
locale to a partial C*-algebra becomes functorial with these
morphisms, let us switch to its external
description~\cite[5.16]{heunenlandsmanspitters:bohrification}:
\begin{equation}
\label{eq:Sonobjects}
  S(A) = \{ F \colon \cC(A) \to \Cat{Set} \mid F(C) \mbox{ open in }
  \Sigma(C), \;F \mbox{ monotone}\}. \\
\end{equation}
For $A$ a partial C*-algebra, $S(A)$ is a locale. We want to extend
this to a functor $S \colon \Cat{PCstar}\op \to \Cat{Loc}$, or
equivalently, a functor $S \colon \Cat{PCstar} \to \Cat{Frm}$. Let $f
\colon A \to B$ be a morphism of partial C*-algebras, $F \in S(A)$,
and $D \in \cC(B)$. If $C \in \cC(A)$ satisfies $\cC f(C) \leq D$,
then we have a morphism $C \to D$ given by the composition $C
\stackrel{f}{\to} \cC f (C) \leq D$. Its Gelfand transform is a
frame morphism $\Sigma(C \stackrel{f}{\to} \cC f(C) \leq D) \colon
\Sigma(C) \to \Sigma(D)$. So, since $F(C) \in \Sigma(C)$, we get an
open in $\Sigma(D)$. The fact that $\Sigma(D)$ is a locale allows us
to take the join over all such $C$, ending up with the candidate
action on morphisms
\begin{equation}
\label{eq:Sonmorphisms}
  Sf(F)(D) = \bigvee_{\substack{C \in \cC(A) \\ \cC f(C) \leq D}} \Sigma(C \to
  \cC f(C) \leq D)(F(C)).
\end{equation}

\begin{theorem}
  Bohrification gives a functor $S \colon
  \Cat{PCstar}\op_{\mathrm{rc}} \to \Cat{Loc}$, where the domain is
  the opposite of the subcategory of $\Cat{PCstar}$ of morphisms
  reflecting commeasurability.
\end{theorem}
\begin{proof}
  One quickly verifies that $Sf(F)$, as given
  in~\eqref{eq:Sonmorphisms}, is monotone, and hence a well-defined
  element of $S(B)$, and that $Sf$ preserves suprema. The greatest
  element $1 \in S(A)$ is also preserved by $Sf$:
  \[
      Sf(1)(D)
    = \bigvee_{\substack{C \in \cC(A) \\ \cC f(C) \leq D}} \Sigma(C \to D)(1)
    = \bigvee_{\substack{C \in \cC(A) \\ \cC f(C) \leq D}} 1
    = 1.
  \]
  The last equality holds because the join is not taken over the empty
  set: there is always $C \in \cC(A)$ with $\cC f(C) \leq D$, namely
  $C=\nul$.

  To finish well-definedness and show that $Sf$ is a frame morphism,
  we need to show that it preserves binary meets. Recall that any
  frame satisfies the infinitary distributive law $(\bigvee_i y_i)
  \land x = \bigvee_i (y_i \land x)$. It follows that one always has
  $(\bigvee_i y_i) \land (\bigvee_j x_j) \geq \bigvee_k (y_k \land
  x_k)$. A sufficient (but not necessary) condition for equality to hold would be if for all $i,j$ there  exists $k$ such that $y_i \land x_j \leq y_k \land x_k$. Expanding
  the definition of $Sf$ and writing $x_C = \Sigma(C \to D)(F(C))$ and
  $y_C = \Sigma(C \to D)(G(C))$ gives precisely this situation:
  \begin{align*}
        Sf(F \land G)(D)
    & = \bigvee_{\cC f(C'') \leq D} x_{C''} \land y_{C''}, \\
        (Sf(F) \land Sf(G))(D)
    & = \big(\bigvee_{\cC f(C) \leq D} x_C \big) \land
        \big(\bigvee_{\cC f(C') \leq D} y_{C'} \big).
  \end{align*}
  So, by Lemma~\ref{lem:reflectcommeas}, $Sf$ will preserve binary meets
  if $f$ reflects commeasurability.
  Finally, it is easy to see that $S(\id) = \id$ and $S(g \after f) =
  Sg \after Sf$.
  \qed
\end{proof}

Let us conclude with four remarks concerning the last theorem.
\begin{itemize}
  \item From the above proof it additionally follows that this choice of
    morphisms is the largest for which the theorem holds:
    $\Cat{PCstar}_{\mathrm{rc}}$ is the largest subcategory of
    $\Cat{PCstar}$ for which~\eqref{eq:Sonmorphisms} gives a well-defined
    frame morphism.

  \item Replacing the Gelfand spectrum by the Stone
    spectrum yields a similar functor $\Cat{PBoolean}_{\mathrm{rc}}\op
    \to \Cat{Loc}$.

  % \item The externalization of an internal locale in a topos
  %   $\mathrm{Sh}(L)$ consists of a locale morphism $L' \to L$, and not
  %   just the locale $L'$ itself. In that spirit, the previous theorem
  %   should have taken into account the map $S(A) \to \cC(A)$ as well. To
  %   consider functorial dependence on $A$ of this map would require to
  %   change the category $\Cat{Loc}$ into something more complicated,
  %   from which we refrain here.

  \item The category $\Cat{Loc}$ is $\Cat{POrder}$-enriched, and
    hence a 2-category.  We remark that the functor $S$ given by
    \eqref{eq:Sonobjects} and \eqref{eq:Sonmorphisms} shows that the
    externalization of Bohrification is a two-dimensional colimit (in
    $\Cat{Loc}$) of the Gelfand spectra of commeasurable
    subalgebras. So, interestingly, whereas the one-dimensional colimit of
    the spectra will often be trivial because of the Kochen-Specker theorem (see the remarks at the end of Section 5), Bohrification shows that a two-dimensional colimit will be
    nontrivial.
\end{itemize}

\subsection*{Acknowledgement}

We are indebted to Steve Vickers, who pointed out
Proposition~\ref{prop:pcstarfunctoriallocaledtopos} to us.
We thank Klaas Landsman for useful comments. Additionally, we
owe the observation concerning~\cite{connes:factornotantiisomorphic}
to Thierry Coquand, and the reference
to~\cite{graetzerkohmakkai:booleansubalgebras} to John Harding and
Mirko Navara and their forthcoming paper `Subalgebras of orthomodular
lattices'.

\bibliographystyle{plain}
\bibliography{colim}

\end{document}